\documentclass[12pt]{article}
\usepackage{amsmath,amsfonts,amssymb,amsthm,amscd}
\title{Generalized Picard-Vessiot extensions and differential Galois cohomology}
\date{\today}
\author{Zo\'e Chatzidakis\thanks{partially supported  by ValCoMo
    (ANR-13-BS01-0006)}{ (CNRS 
    - ENS (UMR 8553))}\and Anand Pillay\thanks{Partially supported by NSF grant DMS-1360702}\ (Notre Dame)}
\newtheorem{Theorem}{Theorem}[section]
\newtheorem{Proposition}[Theorem]{Proposition}
\newtheorem{Definition}[Theorem]{Definition} 
\newtheorem{Remark}[Theorem]{Remark}
\newtheorem{Lemma}[Theorem]{Lemma}
\newtheorem{Corollary}[Theorem]{Corollary}
\newtheorem{Fact}[Theorem]{Fact}

\newtheorem{Question}[Theorem]{Question}

\newcommand{\calu}{{\mathcal U}}
\newcommand{\si}{\sigma}
\begin{document}
\maketitle

\begin{abstract} 
In \cite{Pi2} it was proved that if a differential field $(K,\delta)$ of
charaacteristic $0$ is
algebraically closed and closed under Picard-Vessiot extensions then
every differential algebraic $PHS$ over $K$ for a linear differential
algebraic group $G$ over $K$ has a $K$-rational point (in fact if and
only if).  This paper explores whether and if so, how,  this can be
extended  to (a) several commuting derivations, (b) one automorphism.  Under a
natural notion of ``generalized Picard-Vessiot extension'' (in the case
of several derivations), we give a counterexample.  We also have a
counterexample in the case of one automorphism. We also formulate and
prove some positive statements in the case of several derivations. 

On a montr\'e dans \cite{Pi2} que si un corps diff\'erentiel $(K,\delta)$ de
caract\'eristique $0$
est alg\'ebriquement clos et clos par extensions de Picard-Vessiot,
alors tout espace principal homog\`ene diff\'erentiel alg\'ebrique sur
$K$ a un point $K$-rationnel (et r\'eciproquement). Cet article explore
s'il est possible, et si oui comment, d'\'etendre ce r\'esultat au cas de
(a) plusieurs d\'erivations qui commutent, (b) un automorphisme. Pour
une notion naturelle d'``extension de Picard-Vessiot g\'en\'eralis\'ee''
(dans le cas de plusieurs d\'erivations) nous donnons un
contre-exemple. Nous avons aussi un contre-exemple dans le cas d'un
automorphisme. Enfin, nous formulons et d\'emontrons quelques r\'esultats
positifs dans le cas de plusieurs d\'erivations. 
\end{abstract}

\section{Introduction}
This paper deals mainly with differential fields $(K,\Delta)$ of
characteristic $0$, where $\Delta=\{\delta_1,\ldots,\delta_m\}$ is a set
of commuting derivations on $K$. When $m=1$, the second author showed in
\cite{Pi2} that for a differential field $(K, \delta)$ of characteristic
$0$, the following two conditions are
equivalent: \\
(i) $K$ is algebraically closed and has no proper Picard-Vessiot
extension.\\
(ii) $H^1_{\delta}(K,G)=\{1\}$ for every linear differential algebraic
group $G$ over $K$. (Equivalently, every differential algebraic
principal homogeneous space $X$ for $G$, defined over $K$, has a
$K$-point.)\\[0.05in]
It is natural to ask whether this result generalizes to the case when $m > 1$.  One issue is what are the ``correct" analogues of conditions (i) and (ii) for
several commuting derivations. Condition (ii) is exactly the same and is not controversial. However concerning (i), the Picard-Vessiot theory is a ``finite-dimensional" theory, namely deals with systems of linear equations where the solution set is finite-dimensional, namely has $\Delta$-type $0$. So at the minimum we should include the parametrized Picard-Vessiot ($PPV$) extensions of Cassidy and Singer \cite{CS}.  One of the main points of this paper is to formulate an appropriate notion of ``generalized Picard-Vessiot extension". This, and some variants,  is carried out in Section 3, where we also adapt cohomological arguments of Kolchin. In any case our generalized PV theory
will be a special case of the ``generalized strongly normal theory" from \cite{Pi1} and \cite{Leon-Sanchez-Pillay}, but  still properly include  Landesman's theory \cite{Landesman} (and the so-called parameterized Picard-Vessiot  theory).  So in the case of $m>1$ condition (i) will be replaced by  ``$K$ is algebraically closed and has no proper generalized Picard-Vessiot extensions", (in fact something slightly stronger)   Even with this rather inclusive condition, the equivalence with (ii) will fail, basically due to the existence of proper definable subgroups of the additive group which are orthogonal to all proper definable subfields.  This is carried out in Section 4. In the same section we will give a  positive result (in several derivations) around triviality of $H^{1}_{\Delta}$ and closure under ``generalized strongly normal extensions of linear type".  We will also mention in Section 4, a recent result of Minchenko and Ovchinnikov \cite{MO}, written after the current paper was submitted for publication, which says in effect that definable subgroups of the additive group present the {\em only}  obstruction to generalizing \cite{Pi2} 
to the case $m>1$.

In Section 5, we investigate the context of Picard-Vessiot extensions of  
difference fields (in arbitrary characteristic), and show  that the analogous statement to (i) implies (ii) fails.
hold. Here the problem arises from the existence of proper definable
subgroups of the multiplicative group which are orthogonal to the fixed field. 

\vspace{5mm}
\noindent
The second author would like to thank Michael Singer for several discussions, especially  around the notions of generalized Picard-Vessiot extensions.

\section{Preliminaries}
Our model-theoretic notation is standard. Unless we say otherwise, we work in a saturated model model of a given complete theory. 
The reader is referred to \cite{Tent-Ziegler} for general model theory, stability and simplicity, \cite{Pillay-book} for more on stability, and \cite{Marker} for basics of the model theory of differential fields. 

We will be somewhat freer in our use of model-theoretic notions in this paper compared with say \cite{Pi2}.  The notion of (weak) orthogonality will be important.
Recall that complete types $p(x)$, $q(y)$ over a set $A$  are {\em weakly orthogonal} if $p(x)\cup q(y)$ extends to a unique complete type $r(x,y)$ over $A$.
If $T$ is a simple theory, and complete types $p(x), q(y)$ are over a set $A$, then $p(x)$ and $q(y)$ are said to be {\em almost orthogonal} if whenever $a$ realizes $p$ and $b$ realizes $q$ then $a$ and $b$ are independent over $A$ (in the sense of nonforking). Likewise in the simple environment, $p$ and $q$ (over $A$ again) are said to be {\em orthogonal} if for all $B\supseteq A$ and nonforking extensions $p', q'$ of $p, q$ over $B$, $p'$ and $q'$ are almost orthogonal. 
For $T$ simple, and $p(x), q(y)\in S(A)$ then weak orthogonality of $p$ and $q$ implies almost orthogonality, and conversely almost orthogonality of $p$ and $q$ implies weak orthogonality if at least one of $p$, $q$ is stationary. 

\begin{Remark} Suppose $T$ is stable, $p(x)\in S(A)$, and $\phi(y)$ is a formula over $A$. Then the following are equivalent:
\newline
(i) $p(x)$ and $tp(b/A)$ are weakly orthogonal for all tuples $b$ of realizations of $\phi(y)$.
\newline
(ii) For some (any) realization $c$ of $p$, and any tuple $b$ of realizations of $\phi$, $dcl^{eq}(Ac)\cap dcl^{eq}(Ab) = dcl^{eq}(A)$. 
\end{Remark} 
\begin{proof}  This is well-known, but also proved in Lemma 2.2 (i) of \cite{Leon-Sanchez-Pillay}.

\end{proof}

Another piece of notation we will use is the following: Let  $X$ be a definable set, and $A$ some set of parameters over which $X$ is defined. Then $X_{A}^{eq}$ denotes $dcl^{eq}(X,A)$ and can also be understood as the collection of classes of tuples from $X$ modulo $A$-definable equivalence relations $E$ (as $E$ varies). 

\vspace{5mm}
\noindent
We now pass to differential fields. 
Let $\mathcal U$ be a field of characteristic $0$ with
a set $\Delta=\{\delta_1,\ldots,\delta_m\}$ of $m$ 
  commuting derivations, $m>1$, which is differentially closed.  We assume that $\mathcal U$ is
  sufficiently saturated. We let $\mathcal C$  be the field of absolute
  constants (namely the solution set of $\delta_1(x)=\cdots=\delta_m(x)=0$). We refer to \cite{DAAG}
  and \cite{DAG} for definitions and results in differential algebra. \\
Recall that the theory of $\mathcal U$, denoted DCF$_{0,m}$ is
$\omega$-stable, of  U-rank $\omega^m$. It eliminates quantifiers
and imaginaries, and the definable closure $dcl(A)$ of a subset $A$ of $\mathcal U$ is the
differential field generated by $A$, its algebraic closure $acl(A)$ is
the field-theoretic algebraic closure of $dcl(A)$. Independence is
given by ordinary algebraic independence of the algebraic
closures. See \cite{McG} for proofs. \\
By a differential (or $\Delta$-) subfield of $\mathcal U$ we  mean a subfield closed under the $\delta_{i}$'s. 
If $K$ is a differential subfield of $\mathcal U$ and $a$ a tuple in
  $\mathcal U$, then $K\langle a\rangle $ 
  denotes the differential field
generated by $a$ over $K$. If $L$ is a subfield of $\mathcal U$, then $L^{alg}$
denotes the  algebraic closure of $L$ in $\mathcal U$. \\
If $K$ is a differential subfield of $\mathcal U$, we denote by $K\Delta$ the
Lie algebra of differential  operators defined over $K$, i.e.,
derivations of the form
$D=\sum_{i=1}^ma_i\delta_i$, where the $a_i\in K$. If $y_1,\ldots,y_n$
are indeterminates, then $K\{y_1,\ldots,y_n\}$ (or $K\{y_1,\ldots,y_n\}_\Delta$)  denotes the ring of
$\Delta$-polynomials in the variables $y_1,\ldots,y_n$. 

\begin{Fact} Let  $K$ be  a  $\Delta$-subfield of 
  $\mathcal U$. 
\begin{enumerate}
\item (0.8.13 in \cite{DAG}) (Sit) Let ${\mathfrak A}$ be a perfect $\Delta$-ideal of the
  $\Delta$-algebra $K\{y\}$. A necessary and sufficient condition
  that the set of zeroes ${\mathfrak Z}({\mathfrak A})$ of ${\mathfrak A}$ be a subring of $\mathcal U$, is
  that there exist a vector subspace and Lie subring $\mathcal D$ of
  $K\Delta$ such that ${\mathfrak A}=[{\mathcal D} y]$ (the $\Delta$-ideal generated by
  all 
  $Dy$,  $D\in {\mathcal D}$). When this is the case,
  there exists a commuting linearly independent subset $\Delta'$ of
  $K\Delta$ such that ${\mathfrak Z}({\mathfrak A})$ is the field of absolute constants of the
  $\Delta'$-field $\mathcal U$.



\item (0.5.7 of \cite{DAG}) Recall that a {\em linear $\Delta$-ideal} is
  a $\Delta$-ideal which is generated by homogenous linear
  $\Delta$-polynomials. Let $K$ satisfy the following condition: whenever ${\mathfrak p}$ is a
  linear $\Delta$-ideal of $K\{y_1,\ldots,y_m\}$ with ${\mathfrak p}\cap
  K[y_1,\ldots,y_m]=(0)$, and $$0\neq D\in K[y_1,\ldots,y_m]$$ is
  homogeneous and linear, then ${\mathfrak p}$ has a zero in $K^m$ that is
  not a zero of $D$.  Then every commuting linearly independent subset
  of $K\Delta$ is a subset of a commuting basis of
  $K\Delta$. This in particular happens when $K$ is
  constrained closed. 
\end{enumerate}
\end{Fact}

\begin{Corollary}
\begin{enumerate}
\item  If $K$ is a differential subfield of $\mathcal U$ then the proper $K$-definable subfields of $\mathcal U$ are precisely the common zero sets of (finite) subsets of $K\Delta$.
\item  The definable subfields of $\mathcal U$ have $U$-rank of the form $\omega^{d}$ for some  $0\leq d \leq m$. 

\end{enumerate}
\end{Corollary}
\begin{proof} 1. is clear. For 2.  Let $C$ be a definable subfield of $\mathcal U$. By Fact 2.1 (2), we can find a commuting basis $D_{1},\ldots,D_{m}$ of ${\mathcal U}\Delta$, such that $C$ is the $0$-set of $\{D_{1},\ldots,D_{r}\}$.  But then $({\mathcal U}, +, \times, D_{1},\ldots,D_{m})$ is a model of $DCF_{0,m}$, so as is well-know, in this structure $C$ has $U$-rank $\omega^{m-r}$, so the same is true in ${\mathcal U}$ equipped with the original derivations $\delta_{1},\ldots,\delta_{m}$ (as the two structures are interdefinable).

\end{proof}

\begin{Question} Let $C$ be a $K$-definable subfield of $\mathcal U$. Consider the structure $M$ which has universe $C$ and predicates for all subsets of $C^{n}$ which are defined over $K$ in $\mathcal U$. Does $M$ have elimination of imaginaries?  We know this is the case working over a larger field $K_{1}$, which is enough to witness the interdefinability of $\{\delta_{1},\ldots,\delta_{m}\}$ with the $\{D_{1},\ldots,D_{m}\}$ from the proof of Corollary 1.3 (2). 
\end{Question} 

\begin{Remark} 
\begin{enumerate}
\item  We sometimes call a proper definable subfield of $\mathcal U$ a ``field of constants".
\item Let $F$ be a $K$-definable field. Then $F$ is $K$-definably isomorphic to $\mathcal U$ or to a field of constants. (see \cite{Su2}).

\end{enumerate}
\end{Remark}

\vspace{5mm}
\noindent
In this paper cohomology  will appear in several places. 
Firstly, we have the so-called constrained cohomology set $H^{1}_{\Delta}(K,G)$ where $K$ is a differential subfield of $\mathcal U$, and $G$ is a differential algebraic group over $K$.  This is introduced in Kolchin's second book \cite{DAG}.  $H^{1}_{\Delta}(K,G)$ can be {\em defined} as the set of differential algebraic principal homogeneous spaces over $K$ for $G$, up to differential rational (over $K$) isomorphism.  In \cite{DAG} it is also described  in terms of suitable cocycles from $Aut_{\Delta}(K^{\rm diff}/K)$ to $G(K^{\rm diff})$. Here  $K^{\rm diff}$ is the differential closure of $K$. This is discussed in some detail in the introduction to \cite{Pi2}, Compatibilities with the category of definable groups and $PHS$'s are discussed in the introduction to \cite{Pi2}. See  Fact 1.5(ii) in particular. 

Secondly we have  a  related but distinct theory appearing in Kolchin's earlier book \cite{DAAG}, Chapter VII, Section 8,  which he calls differential Galois cohomology.  This concerns suitable cocycles from the ``Galois group" of a strongly normal extension $L$ of a differential field to $G$ where $G$ is an algebraic group defined over the (absolute) constants of $K$.   In Chapter V of \cite{DAAG} Kolchin also discusses purely algebraic-geometric cohomology theories (although in his own special language), namely Galois cohomology $H^{1}(K,G)$ for $K$ a field and $G$ an algebraic group over $K$  (as in \cite{Serre}) and what he calls $K$-cohomology, $H^{1}_{K}(W,G)$ for $W$ a variety over $K$ and $G$ an algebraic group over $K$.  The interaction between these three cohomology theories is studied  in Chapter V and VI (sections 9 and 10) of \cite{DAAG} and plays an important role in Kolchin's description of strongly normal extensions in terms of so-called $G$-primitives and $V$-primitives.  Generalizing and adapting these notions and work of Kolchin  to a more general Galois theory of differential fields will be the content of the proof of Proposition 3.8 below.

In the title of the current paper we use the expression ``differential Galois cohomology" to refer both to differential Galois cohomology in the sense of \cite{DAAG} and constrained cohomology in the sense of \cite{DAG}.

\section{Generalized Picard-Vessiot extensions and variants}
We are still working in the context of a saturated differentially closed field $\mathcal U$ with respect to the set $\Delta = \{\delta_{1},\ldots,\delta_{m}\}$ of derivations. 
We first recall a definition from \cite{Leon-Sanchez-Pillay} (Definition 3.3) which is itself a slight elaboration on a notion from \cite{Pi1}. 

\begin{Definition} Let $K$ be a (small) subfield of $\mathcal U$, $X$
  some $K$-definable set, and $L$ a differential field extension of $K$
  which is finitely generated over $K$ (as a differential field).  $L$
  is said to be an $X$-strongly normal extension of $K$ if (i) for any
  $\sigma\in Aut({\mathcal U}/K)$, $\sigma(L)\subseteq L\langle
  X\rangle$, and (ii) $K\langle X\rangle \cap L = K$.
\end{Definition}

\begin{Remark} In the context of of Definition 3.1, let $L = K\langle
  b\rangle$ and let $q = tp(b/K)$. Then (i) says that for any
  realization  $b_{1}$ of $q$, $b_{1}\in dcl(K,b, X)$, and (ii) says
  that $q$ is weakly orthogonal to $tp(a/K)$ for any tuple $a$ of elements of $X$. (See Remark 2.1.)   Moreover, in this situation (i.e  when  (i) and (ii) hold), the type $q$ is isolated (see Lemma 2.2 (ii) of \cite{Leon-Sanchez-Pillay}). 
\end{Remark} 

We will need to know something about the Galois group associated to an $X$-strongly normal extension. This is contained in Theorem 2.3 of \cite{Leon-Sanchez-Pillay}. But we give a summary.   So we assume $L$ is an $X$-strongly normal extension of $K$ and we use notation from Remark 3.2.   In particular we are fixing $b$ such that $L= K\langle b\rangle$. Let $Q$ be the set of realizations of the (isolated) type $q = tp(b/K)$.  Then
$Q$ is a $K$-definable set, which moreover isolates a complete type over $K\langle X\rangle$.  
Let $Aut(Q/K, X)$ be the group of permutations of $Q$ induced by automorphisms of $\mathcal U$ which fix pointwise $K$ and $X$.  Then
\begin{Fact}  $Aut(Q/K,X)$ acts regularly (i.e. strictly transitively) on $Q$.  In other words  $Q$ is a principal homogeneous space for $Aut(Q/K,X)$ (under the natural action). 
\end{Fact}

\begin{Fact} There is a definable group $G$, living in $X_{K}^{eq}$, and defined over $K$, $K$-definable surjective functions $f: Q\times G \to Q$ and  $h:Q\times Q \to G$, and an isomorphism $\mu:Aut(Q/K,X) \to G$ with the following properties:
\begin{enumerate}
\item For $b_{1}, b_{2}\in Q$ and $g\in G$, $h(b_{1},b_{2}) = d$ iff $f(b_{1}, d) = b_{2}$. 
\item  For each $\sigma\in Aut(Q/K,X)$, $\mu(\sigma) = h(b,\sigma(b))$ (equals the unique $d\in G$ such that $f(b,d) = \sigma(b)$).  
\item for $b_{1}, b_{2}, b_{3}\in Q$, $h(b_{1},b_{2})\cdot h(b_{2}, b_{3}) = h(b_{1}, b_{3})$
\item The group operation of $G$ is: $d_{1}\cdot d_{2} = h(b,f(f(b,d_{1}), d_{2}))$
\item The action of $G$ on $Q$ (induced by the isomorphism $\mu$) is   $d\cdot b_{1} = f(f(b,d),h(b,b_{1}))$.
\end{enumerate}

\end{Fact}

\begin{Remark}  (i)  On the face of it, $G$ and its group structure are defined over $K\langle b\rangle$, but as $dcl(K,b)\cap X$ is contained in $K$, it is defined over $K$.  On the other hand the action of $G$ on $Q$ DOES need the parameter $b$.  The $G$ we have described is the analogue of the ``everybody's Galois group" from the usual strongly normal  theory (where it is an algebraic group in the absolute constants). In any case, with a different choice of $h$ and $f$ (but the same $b$)  would give the same $G$ up to $K$-definable isomorphism.
\newline
(ii)  In the above we have worked in an ambient saturated model $\mathcal U$. But we could have equally well worked with the differential closure of $K$  (in which $L$ lives) in place of $\mathcal U$. 
\end{Remark}

\begin{Definition} 

\begin{enumerate} 

\item Let $K$ be a differential field, and $X$ a $K$-definable set. We call $L$ an $X$-strongly normal extension of linear type, if (i) $L$ is an $X$-strongly normal extension of $K$, and (ii) The Galois group $G$ as in Fact 3.4  $K$-definably embeds in $GL_{n}(\mathcal U)$. 

\item  Let $K$ be a differential field and $C$ a $K$-definable field of constants. We call $L$ a generalized $PV$ extension of $K$ with respect to $C$ if (i) $L$ is a $C$-strongly normal extension of $K$, and (ii) the Galois group $G$ from 3.4 $K$-definably embeds in $GL_{n}(C)$  (some $n$). 

\end{enumerate}

\end{Definition}

\begin{Remark}
\begin{enumerate}
\item  Note that taking $X=C$, (ii) is stronger than (i), since we
  impose that the Galois group $K$-definably embeds into $GL_n(C)$, and
  not only $GL_n({\mathcal U})$.  Could we replace Definition 3.6.2  by simply ``$L$ is a
  $C$-strongly normal extension of linear type"?  The issue includes the following: Suppose the Galois group $G$ definably embeds in $GL_{n}({\mathcal U})$, and is also in $dcl(C,K)$. Does $G$ $K$-definably embed in $GL_{n}(C)$? 
\item According to the classical theory (\cite{DAAG}) a Picard-Vessiot extension $L$ of $K$ is a differential field extension of $K$  generated by a solution $B$ of a system  $\delta_{1}Z = A_{1}Z$,\dots,$\delta_{m}Z = A_{m}Z$, where $Z$ ranges over $GL_{n}$, each $A_{i}$ is an $n$ by $n$ matrix over $K$, the $A_{i}$ satisfy the Frobenius (integrability) conditions, AND $C_{L} = C_{K}$ where $C_{K}$ denotes ${\mathcal C}\cap K$ etc.  So we see easily that this is an example of a generalized $PV$-extension.
\item  The so-called parameterized Picard-Vessiot theory in \cite{CS}
  gives another example of a generalized PV extension of $K$ (with
  respect to the field of $\delta_1$-constants).  Here we consider a single linear differential equation $\delta_{1}Z = AZ$ where again $Z$ ranges over $GL_{n}$ and $A$ is an $n$ by $n$ matrix over $K$. A $PPV$ extension $L$ of $K$ for such an equation is a $\Delta$-extension $L$ of $K$ generated by a solution $B$ of the equation such that $K$ and $L$ have the same $\delta_{1}$-constants.   This is put in a somewhat more general context in \cite{HS}.
\end{enumerate}
\end{Remark}

\begin{Proposition} Suppose that $K$ is a differential field, $X$ a $K$-definable set, and $L = K\langle b\rangle$ an $X$-strongly normal extension of $K$ of linear type, with Galois group $G < GL_{n}(\mathcal U)$.  Let $\mu$ be the isomorphism between $Aut(Q/K,X)$ and $G$ as in Fact 3.4. Then there is $\alpha\in GL_{n}(L)$ such that $\mu(\sigma) = \alpha^{-1}\sigma(\alpha)$ for all $\sigma\in Aut(Q/K,X)$. 
\end{Proposition}
\begin{proof} This is an adaptation of the proof of Corollary 1, Chapter
  VI,  from \cite{DAAG}.  When $X(K) = X(K^{\rm diff})$ (which was part of the definition of $X$-strongly normal extension in the original paper \cite{Pi1}), and $K$ is algebraically closed, it is easier, and is Proposition 3.4 of \cite{Pi1} (in the one derivation case which extends easily to several derivations).

We use the objects and notation in Fact 3.4. As $G$ is a subgroup of $GL_{n}$, we can consider $h$ as a $K$-definable function from $Q\times Q \to GL_{n}$.  
We already have $q = tp(b/K)$.  Let now $q_{0}$ be $tp(b/K)$ in $ACF$. Let  $n$ be the number of types $tp(b,c/K)$ in $DCF_{0}$  where $c$ realizes $q$ and $c$ is independent from $b$ (in the sense of $DCF_{0}$). Likewise let $n_{0}$ be the number of types $tp(b,c/K)$ in $ACF_{0}$ where $c$ realizes $q_{0}$ and is independent from $b$ over $K$ in $ACF_{0}$. 

\vspace{2mm}
\noindent
{\em Claim.}  After replacing $b$ by some larger finite tuple in $L$ we may assume that (i) $h(-,-)$ is a $K$-rational function (rather than $K$-differential rational), and (ii)  $n = n_{0}$. 
\newline
{\em Proof.}  We know that  if $L_{1}, L_{2}$ are differential fields containing $K$ then the differential field generated by $L_{1}$ and $L_{2}$ coincides with the field generated by $L_{1}$ and $L_{2}$.   So we can apply compactness to the implication: if $b,c$ realize $q$ then $h(b,c)$ is contained in the field generated by $K\langle b\rangle$ and $K\langle c\rangle$ to obtain the required conclusion (i).  Clearly we can further extend $b$ so as to satisfy (ii). 

\vspace{2mm}
\noindent
Let $W$ be the (affine) algebraic  variety over $K$ whose generic point is $b$, i.e. whose generic type is $q_{0}$. Then by Step I we have the $K$-rational function $h(-,-)$ such that $h(b_{1}, b_{2})$ is defined whenever $b_{1}, b_{2}$ are independent realizations of $q_{0}$. Moreover if $b_{1}$, $b_{2}$, $b_{3}$ are  independent realisations of $q_{0}$ then $h(b_{1},b_{2})\cdot h(b_{2},b_{3}) = h(b_{1}, b_{3})$ (using 3.4.3).  We now refer to Chapter V, Section 17 of \cite{DAAG} which is just about algebraic varieties, and by Proposition 24 there, there is a Zariski-dense, Zariski open over $K$ subset $U$ of $W$ such that $h$ extends to a $K$-rational function to $GL_{n}(\mathcal U)$, which we still call $h$, which is defined on $U\times U$, satisfies the cocycle condition, and  $h(u,u) = 1$ for all $u\in U$ and $h(v,u) = h(u,v)^{-1}$.

Let us now pick $u\in U(K^{alg})$, and define $h_{u}:Gal(K^{alg}/K) \to
GL_{n}(K^{alg})$ by  $h_{u}(\sigma) = h(u,\sigma(u))$ for $\sigma\in Gal(K^{alg}/K)$.  Then by Theorem 12 of Chapter V of \cite{DAAG}, $h_{u}\in H^{1}(K,GL_{n})$, namely is continuous and satisfies the condition $h_{u}(\sigma\tau) = h_{u}(\sigma)\tau(h_{u}\tau)$.  As $H^{1}(K,GL_{n})$ is trivial, the cocycle $h_{u}$ is ``trivial" namely there is $x\in GL_{n}(K^{alg})$ such that $h_{u}(\sigma) = x^{-1}\sigma(x)$ for all $\sigma\in Gal(K^{alg}/K)$.  As in the proof of part (b) of Theorem 12, Chapter V, \cite{DAAG} there is a $K$-rational function $g:W\to GL_{n}$ such that  $h(u,v) = g(u)^{-1}g(v)$ for all $(u,v)\in dom(h)$.  Returning to the differential algebraic setting, it follows that 
$h(b,\sigma(b)) = g(b)^{-1}\cdot g(\sigma(b))^{-1}$ for all $\sigma\in Aut(Q/K,X)$.  Let $g(b) = \alpha\in GL_{n}(L)$. Then $g(\sigma(b)) = \sigma(g(b)) = \sigma(\alpha)$ for all 
$\sigma\in Aut(Q/K,X)$, as required. 
\end{proof} 

\begin{Corollary} Under the same assumptions as the proposition above, there is $\alpha\in GL_{n}(L)$ such that $K\langle\alpha\rangle = L$ and the coset $\alpha\cdot G$ is defined over $K$.  (Hence $L$ is a generated by a solution of the ``PDE" $\nu(z) = a$  where $\nu$ is the $K$-definable map from $GL_{n}(\mathcal U)$ to $W = GL_{n}({\cal U})/G$, and $a\in W(K)$. )
\end{Corollary}
\begin{proof} Let $\alpha\in GL_{n}(L)$ be as in the conclusion of Proposition 3.8 above.  Then for each $\sigma\in Aut(Q/K,X)$, $\sigma(\alpha)\in \alpha\cdot G$, whereby $\alpha\cdot G$ is fixed (setwise) by $Aut(Q/K,X)$. This implies that $\alpha\cdot G$ is defined over $K,X$. But it is also clearly defined over $L$. Hence, by our assumptions, it is defined over $K$. 
On the other hand, if $\sigma\in Aut(Q/K,X)$ fixes $\alpha$ then $\mu(\sigma) = 1\in G$, so $\sigma$ is the identity, so fixes $L$. So $L= K\langle\alpha\rangle$. 
\end{proof} 

\begin{Remark} The obvious analogies of Proposition 3.8 and Corollary 3.9 hold if the Galois group $G$ is assumed to $K$-definably embed in an algebraic group $H$ for which $H^{1}(K,H) = \{1\}$. 
\end{Remark} 

We now want to get somewhat more explicit information when $L$ is a generalized $PV$ extension of $K$. The following easy lemma which is left to the reader  tells how to eliminate the interpretable set $GL_{n}({\mathcal U})/GL_{n}(C)$ when $C$ is a definable subfield of $\mathcal U$.   
\begin{Lemma} Let $C$ be a ``field of constants" defined by the set of common zeros of definable derivations $D_{1},\ldots,D_{r}$. Fix $n$.  Let $b_{1}, b_{2}\in GL_{n}(\mathcal U)$. Then $b_{1}GL_{n}(C) = b_{2}GL_{n}(C)$ iff  $D_{i}(b_{1})\cdot b_{1}^{-1} = D_{i}b_{2}\cdot b_{2}^{-1}$ for $i=1,\ldots,r$. 
\end{Lemma}

\begin{Corollary}  Let $K$ be a differential field subfield of $\mathcal U$, and let $C$ be a $K$-definable field of constants, defined by the common zero set of $D_{1},\ldots,D_{r}$ where the $D_{i}\in K\Delta$.  Let $L$ be a differential field extension of $K$. Then the following are equivalent:
\newline
(i) $L$ is a generalized Picard-Vessiot extension of $K$ with respect to $C$,
\newline
(ii) for some $n$ there are $n$ by $n$ matrices $A_{1}$,\dots,$A_r$ over $K$, and a solution $\alpha\in GL_{n}(\mathcal U)$ of the system $D_{1}Z = A_{1}Z$,\dots, $D_{r}Z = A_{r}Z$, such that $L = K\langle \alpha \rangle$,  and $K\langle \alpha\rangle  \cap K\langle C\rangle = K$. 
\end{Corollary}
\begin{proof} (ii) implies (i).  First note that the set of common solutions of the system $D_{1}Z = 
A_{1}Z$, \dots, $D_{r}Z = A_{r}Z$ is precisely the coset $b\cdot GL_{n}(C)$ inside $GL_{n}(U)$.  Letting $q = tp(b/K)$,  every realization $b_{1}$ of $q$ is of the form $bc$ for some $c\in GL_{n}(C)$.  So bearing in mind the second part of (ii), $L = K\langle b\rangle$ is a $C$-strongly normal extension of $K$.  Let $Q$ be the set of realizations of the type $q$ (which we know to be isolated and moreover implies  a complete type over $K, C$).  Taking $h(b_{1}, b_{2})$ to be $b_{1}^{-1}b_{2}$ (from $Q\times Q \to GL_{n}(C)$), and $f(b_{1},c)$ to be $b_{1}c$, we see that the Galois group $G$ from Fact 3.4 is a $K$-definable subgroup of $GL_{n}(C)$. Hence $L$ is a generalized $PV$ extension of $K$ with respect to $C$.

\vspace{2mm}
\noindent
(i) implies (ii).  Assume that $L = K\langle b\rangle$ is a generalized $PV$ extension of $K$ with respect to $C$. So we may assume that the Galois group $G$ is a subgroup of $GL_{n}(C)$. Let $\alpha\in GL_{n}(L)$ be given by Corollary 3.9, namely $L = K\langle\alpha\rangle$ (so $tp(\alpha/K)$ implies $tp(\alpha/K,C)$)  and the coset $\alpha\cdot G$ is defined over $K$.  But then the coset $\alpha\cdot GL_{n}(C)$ is also defined over $K$. Let $A_{i} = D_{i}(\alpha)\alpha^{-1}$ for $i=1,\ldots,r$. By Lemma 3.11, each $A_{i}$ is an $n\times n$ matrix over $K$, so (ii) holds. 
\end{proof} 

The following follows from the proof of Proposition 4.1 of \cite{Pi1} and is essentially a special case of that proposition. 
\begin{Remark}  (i) Let $F_{0}$ be a differentially closed $\Delta$-field, $C$ an $F_{0}$-definable field of constants, and $G$ an $F_{0}$-definable subgroup of $GL_{n}(C)$. Then there are differential fields $F_{0} < K< L$ such that $L\cap F_{0}\langle C\rangle =  F_{0}$, and $L$ is a generalized Picard-Vessiot extension of $K$ with respect to $C$ with Galois group $G$. 
\newline
(ii) $K_{0}$ be a differential field (subfield of $\mathcal U$) which is  contained in the field of $\delta_{1}$-constants $C$ of $\mathcal U$ and is differentially closed as a $\Delta\setminus\{\delta_{1}\}$ field. Let $G$ be a subgroup of $GL_{n}(C)$ defined over $F_{0}$. Then there are again differential subfields $F_{0}< K <L$ of $\mathcal U$, such that
$C\cap L = C_{0}$, and $L$ is a $PPV$ extension of $K$ with Galois group $G$. 
\end{Remark}

The reader should note that in the remark above,  the base  $K$ is generated (over $F_{0}, K_{0}$, respectively) by something which is $\Delta$-transcendental. 

\section{Main results}
Here we prove both the positive and negative results around trying to extend \cite{Pi2} to the context of several commuting derivations. 
\begin{Definition} (i) $K$ is closed under generalized strongly normal extensions of linear type, if for no $K$-definable set $X$ does $K$ have a proper $X$-strongly normal extension of linear type.
\newline
(ii) $K$ is strongly closed under generalized strongly normal extensions of linear type if for every $K$-definable subgroup $G$ of $GL_{n}(\mathcal U)$ and $K$-definable coset $Y$ of $G$ in $GL_{n}(\mathcal U)$, there is $\alpha\in Y\cap GL_{n}(K)$.
\newline
(iii) $K$ is closed under generalized $PV$-extensions if $K$ has no proper generalized $PV$-extension.
\newline
(iv) $K$ is strongly closed under generalized $PV$-extensions if for
every $K$-definable field of constants $C$ defined by derivations
$D_{1},\ldots,D_{r}$, and consistent system $D_{1}Z =
A_{1}Z,\ldots,D_{r}Z = A_{r}Z$ where the $A_{i}$ are over $K$, there is
a solution $\alpha\in GL_{n}(K)$. 
\end{Definition}
Note that we could restate (iv) in the same way as (ii):  every $K$-definable coset $Y$ of $GL_{n}(C)$ in $GL_{n}(\mathcal U)$ has a $K$-point. 

\begin{Remark} (i) In the definition
 above, (ii) implies (i) and (iv) implies (iii).
\newline  
(ii) We might also want to define $K$ to be very strongly closed under generalized $PV$-extensions, if we add to (iii) or equivalently (iv), that for every $K$-definable field $C$ of constants $C(K) = C(K^{\rm diff})$.
\end{Remark}

A first attempt at generalizing  \cite{Pi2}  to several commuting
derivations might be to ask whether $K$ is algebraically closed and
(strongly) closed under generalized $PV$ extensions iff
$H^{1}_{\Delta}(K,G)$ is trivial for all linear  {differential} algebraic groups defined over $K$.  Right implies left is trivial. But the following gives a counterexample to left implies right when $m=2$. In this case we can simply quote results of Suer. We will point out the extension to arbitrary $m> 1$ later. 
\begin{Fact} ($m = 2$.) Let $G$ be the subgroup of $({\mathcal U},+)$
  defined by $\partial_{1}(y) = \partial_{2}^{2}(y)$ (the so-called heat
  variety). 
\begin{enumerate}
\item Then the $U$-rank of $G$ is $\omega$ and the generic type  of
  $G$ is orthogonal to $stp(c/A)$ where $c$ is any tuple from an
  $A$-definable field of constants (i.e. proper $A$-definable subfield
  of $\mathcal U$). 
\item Let $p$ be a type (over some differential field $K$) realised in $(\calu,+)$, and assume
  that if $a,b$ are independent realisations of $p$ over $K$, then $a-b$
  realises the generic type  
  of $G$ over $K$. If $C$ a $K$-definable field of constants, then $p$ is orthogonal to $tp(c/K)$ for any
  $c$ generating a $C$-strongly
 normal extension of $K$ of linear type.  
\end{enumerate} 
\end{Fact}
\begin{proof} 1. Let $p_{0}$ be the generic type of $G$ (a stationary type over $\emptyset$). By Proposition 4.5 of \cite{Su2}, $p_{0}$ has $U$-rank $\omega$. By Proposition 4.4, Proposition 3.3, and Theorem 5.6 of \cite{Su2},  $p_{0}$ is orthogonal to the generic 
type of any definable subfield $F$ of $\mathcal U$  and is
  orthogonal to any type of rank $<\omega$. Hence it is orthogonal to
any type realised in a field of constants.

2. Over any realisation of $p$, there is a definable bijection
between the set of realisations of $p$ and the elements of
$G$. Similarly, there is a $c$-definable bijection between the set of
realisations of $tp(c/K)$ and some definable subgroup of $GL_n(K\langle
C\rangle)$. 
As $p_0$ is orthogonal to all types realised in constant
subfields of $\calu$, we obtain that $p$ and $q$ are orthogonal.
\end{proof}

\begin{Proposition}  ($m = 2$.)   There is a differential subfield $K$ of $\mathcal U$ which is algebraically closed and very strongly closed under generalized $PV$ extensions, but such that  $H^{1}_{\Delta}(K,G) \neq \{1\}$ where $G$ is the heat variety above. 
\end{Proposition} 
\begin{proof} 
 Let $Q$ denote the operator $\delta_{1} - \delta_{2}^{2}$.  Then $Q$ is
 a $\emptyset$-definable surjective homomorphism  from $({\mathcal U},
 +)$ to itself, whose fibres are the cosets of $G$ in $({\mathcal
   U},+)$.  Choose generic $\alpha\in \mathcal U$, namely
 $U(tp(\alpha/\emptyset)) = \omega^{2}$, and let $d = Q(\alpha)$.  Then
 $d$ is also generic in ${\mathcal U}$ (by the $U$-rank inequalities for
 example). Let $X_d = Q^{-1}(d)$ (the solution set of $Q(y)=d$), and $p = tp(\alpha/d)$.  The following is basically a repetition of Claim 1 in the proof of Proposition 4.1 in \cite{Pi1}:
\newline
{\em Claim.} $\alpha$ realizes the generic type of $X_{d}$, and is moreover isolated by the formula $y\in X_{d}$.
\newline
{\em Proof.}  If $\beta\in X_{d}$ then $\beta$ is generic in $\mathcal U$ so $tp(\beta) = tp(\alpha)$. But $Q(\beta) = Q(\alpha) = d$, hence $tp(\beta/d) = tp(\alpha/d)$. As $\beta$ could have been chosen to be generic in $X_{d}$, the claim is proved. 

\vspace{2mm}
\noindent
By the Claim and Fact 4.3,
\newline
$(*)$ for any $A\subset {\mathcal U}$, if $C$ is an $A$-definable field
of constants and $c$ is a tuple of elements of $C$, then $p$ is orthogonal to $tp(c/A)$. Furthermore, $p$ is
  orthogonal to any $tp(b/A)$ with $A\langle b\rangle$ $C$-strongly
  normal over $A$. 
\newline
Let
$K_{0}$ be the algebraic closure of the differential subfield of
$\mathcal U$ generated by $d$ (namely $K_{0} = acl(d)$). So $p$, being
stationary, has a unique extension over $K_{0}$.  Fix a differential
closure $K_{0}^{\rm diff}$ of $K_{0}$. For each $K_{0}$-definable field of
constants $C$ and consistent system $D_{1}Z = A_{1}Z,\ldots,D_{r}Z =
A_{r}Z$ over $K_{0}$ (where $C$ is defined as the the zero set of
$D_{1},\ldots,D_{r}$), adjoin to $K_{0}$, both $C(K_{0}^{\rm diff})$ and a
solution of the system in $GL_{n}(K_{0}^{\rm diff})$. Let $K_{1}$ be the
algebraic closure of the resulting differential field. Then
$K_{0}^{\rm diff} = K_{1}^{\rm diff}$, and by $(*)$, $p$ implies a unique type
over $K_{1}$.  Build $K_{2}$, $K_{3}$,\dots\ similarly and let $K$ be
the union. Then by construction $K$ is algebraically closed and  very
strongly closed under generalized $PV$-extensions, but $p$ isolates a
unique complete type over $K$, which precisely means that $X_{d}$, a $PHS$ for the linear differential algebraic group $G$, has no solution in $K$. 

\end{proof} 

For $m>2$ we can extend the results of \cite{Su2} to show that the
subgroup $G$ of the additive group defined by $\partial_{1}(y)
= \partial_{2}^{2}(y)$   is orthogonal to
any definable field of constants. Indeed, the computations made in
Theorem II.2.6 of \cite{DAAG} (taking $E=\{\delta_2^2\}$) give that $G$ has $\Delta$-type $m-1$ and
typical $\Delta$-dimension $2$. Hence the generic of $G$ is orthogonal
to all types of $\Delta$-type $<m-1$, and to all types of $\Delta$-type
$m-1$ and typical $\Delta$-dimension $1$ (Proposition 2.9 and Theorem
5.6 in
\cite{Su2}). In particular it is orthogonal to any type realised in a
``field of constants'' (Proposition 3.3 of \cite{Su2}).   
Proposition 4.4  readily  adapts, giving a counterexample for any $m>1$ to ``natural" analogues of the main results of \cite{Pi2}.

We now give our  positive result, which is close to being tautological. 

\begin{Proposition} Let $K$ be a differential field. Then the following are equivalent. 
\newline
(i) $K$ is algebraically closed and strongly closed under generalized strongly normal extensions of linear type (i.e. for any $K$-definable subgroup $G$ of $GL_{n}(\mathcal U)$ of $U$-rank $< \omega^{m}$, and $K$-definable coset $X$ of $G$  in 
$GL_{n}({\mathcal U})$, $X\cap GL_{n}(K) \neq \emptyset$). 
\newline
(ii) ${H^{1}_\Delta}(K,G) = \{1\}$ for any linear differential algebraic group $G$ defined over $K$.
\end{Proposition}
\begin{proof} (i) is a special case of (ii).  (If $K$ is not algebraically closed then a Galois extension is generated by an algebraic $PHS$ $X$ for a finite (so linear) group $G$ over $K$, where $X$ has no $K$-point. And of course any $K$-definable coset of a $K$-definable subgroup $G$ of $GL_{n}({\mathcal U})$ is a special case of a $K$-definable $PHS$ for the linear differential algebraic group $G$.)
\newline
(i) implies (ii):  First as $K$ is algebraically closed, (ii) holds for
finite $G$. So by the inductive principle  (if
  $1\to N\to G\to H\to 1$ is a short exact sequence of differential
  algebraic groups over $K$, and $H^1_\Delta(K,N)=H^1_\Delta(K,H)=\{1\}$,
  then $H^1_{\Delta}(K,G)=\{1\}$. See Lemma 2.1 of \cite{Pi2}  for a proof when
  $|\Delta|=1$)  
 we may assume that $G$ is connected. 
Now first assume that $U(G) < \omega^{m}$. And suppose $(G,X)$ is a
$K$-definable $PHS$. We claim that $(G,X)$
embeds in an  algebraic $PHS$ $(G_{1},X_{1})$ over $K$.  This is stated in fact 1.5 (iii) of \cite{Pi2} with
reference to \cite{Pillay-foundations}. But we should be more precise. What we prove in \cite{Pillay-foundations} is that the differential algebraic group $G$ embeds in an algebraic group over $K$. But this, together with Weil's proof \cite{Weil}, easily adapts to the (principal) homogeneous space context. 

By Lemma 4.7 of \cite{Pillay-foundations}, and the linearity of $G$ we may assume that $G_{1}$ is linear (namely a $K$-algebraic subgroup of $GL_{n}(\mathcal U)$). As $K$ is algebraically closed $(G_{1},X_{1})$ is isomorphic to $(G_{1},G_{1})$ over $K$, and this isomorphism takes $X$ to a coset of $G$ in $G_{1}$ defined over $K$. By (i) $X$ has a $K$-point. 

So  by the inductive principle, we may assume that $G$ is ``$m$-connected", namely has no proper definable subgroup $H$ with $U(G/H) < \omega^{m}$.  Then the last part of the proof of Theorem 1.1 in \cite{Pi2}, more precisely the proofs of Case 2(a) and 2(b), work to reduce to the situation when $G$ is algebraic and we can use Kolchin's Theorem (Fact 1.6 in \cite{Pi2}). 

\end{proof}

Finally we will mention, as requested by the referee,  a nice related result \cite{MO}, which we became aware of after the current paper was submitted, and which does in a sense give a satisfactory extension of \cite{Pi2} to the case of several commuting derivations, although we did not check the proofs in detail.  The content of the main result of \cite{MO}, using notation of the current paper is that the  following are equivalent, in the context of differential fields $(K,\Delta)$ where $\Delta = \{\partial_{1}, \ldots, \partial_{m}\}$ is a finite set of commuting derivations:
 \newline
\noindent
(i) $K$ is algebraically closed, strongly closed under generalized $PV$-extensions, and $H^{1}_{\Delta}(K,G)$ is trivial for any differential algebraic subgroup $G$ of the additive group, defined over $K$. 
\newline
(ii) $H^{1}_{\Delta}(K,G)$ is trivial, for {\em any} linear differential algebraic group defined over $K$. 

So this says that the main theorem of \cite{Pi2} does extend naturally  to  several commuting derivations, modulo the case of principal homogeneous spaces for differential algebraic subgroups of the additive group.   The proof (of (i) implies (ii)) in \cite{MO} essentially follows the general line of the inductive argument of \cite{Pi2}, reducing to the cases where $G$ is  finite, a connected subgroup of the multiplicative group, a connected subgroup of the additive group, or noncommutative simple. The additional assumption in (i) deals with the third case and allows the arguments to go through (with some additional complications and use of results in the literature). 

\section{The difference case}
If $K$ is a field equipped with an automorphism $\sigma$, then by  a linear difference equation over $(K,\sigma)$ we mean something of the form $\sigma(Z) = A\cdot Z$ where $A\in GL_{n}(K)$ and $Z$ is an unknown nonsingular matrix.  When it comes to a formalism for difference equations, Galois theory, etc. there is now a slight discrepancy between algebraic and model-theoretic approaches.  In the former case, difference rings $(R,\sigma)$, which may have zero-divisors, enter the picture in a fundamental way. The latter, on the other hand, is field-based, where the difference fields considered are difference subfields of a ``universal domain" $({\mathcal U},\sigma)$, a model of a certain first order theory $ACFA$ (analogous to $DCF_{0}$).  In this section we will opt for the model-theoretic approach.  Papers such as \cite{CHS} and \cite{Kamensky} discuss differences and compatibilities between the treatments of the Galois theory of linear difference equations in the two approaches.   But we will not actually need to engage with delicate issues regarding Picard-Vessiot extension of difference fields (or rings).

\begin{Definition} Let $(K,\sigma)$ be a difference field. We will say that $(K,\sigma)$ is strongly $PV$-closed if every linear difference equation $\sigma(Z) = A\cdot Z$ over $K$ has a solution in $GL_{n}(K)$.
\end{Definition}

The theory $ACFA$ is the model companion of the theory of fields
equipped with an automorphism, in the language of unitary rings together
with a symbol for an automorphism. See the seminal paper \cite{CH}
which, among other things,  describes the completions of $ACFA$,  its
relative quantifier elimination, and its ``stability-theoretic"
properties (it is unstable but supersimple).  We fix a saturated model
$({\mathcal U}, \sigma)$ of $ACFA$. $F, K, L \ldots$ denote (small)
difference subfields of $\mathcal U$.  By  a difference polynomial
$P(x_{1},\ldots,x_{n})$ over $K$ we mean a polynomial over $K$ in
indeterminates $x_{1},\ldots,x_{n}, \sigma(x_{1}),\ldots,$ $\sigma(x_{n}),
\sigma^{2}(x_{1}),\ldots,\sigma^{2}(x_{n})$ \dots .  By a difference-algebraic  variety (defined over $K$) we mean a subset of some ${\mathcal U}^{n}$ defined by a (finite) set of difference polynomials over $K$.  If $V$ and $W$ are two such difference-algebraic  varieties over $K$ then a difference-algebraic  morphism over $K$ from $V$ to $W$ is a map whose coordinates are given by difference polynomials over $K$. So we have a category of (affine) difference-algebraic varieties. We may just say ``difference variety". 
\begin{Definition}
\begin{enumerate}  
\item   By a linear difference algebraic group  (or just linear difference group) defined over $K$ we mean a subgroup of some $GL_{n}(\mathcal U)$ whose underlying set is a difference algebraic set over $K$.
\item If $G$ is a linear difference algebraic group over $K$, then a difference algebraic $PHS$ over $K$ for $G$ is a difference algebraic variety $X$ over $K$ together with a difference morphism over $K$, $G\times X \to X$ giving $X$ the structure of a $PHS$ for $G$.
\item   If $X$ is such a difference algebraic $PHS$ for $G$ over $K$ we say that $X$ is trivial if $X(K) \neq\emptyset$. 
\end{enumerate} 

\end{Definition}

\begin{Remark}
\begin{enumerate}
\item We should not confuse difference-algebraic varieties and groups with algebraic $\sigma$-varieties and groups from \cite{KP}, which are objects belonging to algebraic geometry.
\item   We have not formally defined $H^{1}_{\sigma}(K, G)$ for a linear
  difference-algebraic group over $K$, partly because there are various
  other choices of what category to work in, such as the category of
  definable $PHS$'s \dots 
\end{enumerate} 
\end{Remark}

\begin{Fact} Let $K$ be a difference subfield of $\mathcal U$, and $a$ a tuple (in
  $\mathcal U$) such that $\sigma(a)\in K(a)^{alg}$. One defines the
  {\em limit degree of $a$ over $K$} by 
$${\rm ld}(a/K)=\lim_{n\rightarrow \infty}[K(a,\ldots,\si^{n+1}(a)):K(a,\ldots,\si^n(a))]$$
and the {\em inverse limit degree of $a$ over $K$} by 
$${\rm ild}(a/K)=\lim_{n\rightarrow \infty}
[K(a,\ldots,\si^{n+1}(a)):K(\si(a),\ldots,\si^{n+1}(a))].$$
Then these degrees are multiplicative in tower (see \cite{Co} section
5.16), e.g. ${\rm ld}(ab/K)={\rm ld}(a/K){\rm ld}(b/K(a,\si(a),\si^2(a),\ldots))$. Observe that if $b$ is
algebraic over $K$, then ${\rm ld}(b/K)={\rm ild(b/K)}$. Hence, setting $\Delta(a/K):=\frac{{\rm ld}(a/K)} {{\rm
    ild}(a/K)}$, if $b\in
K(a)^{alg}$, then $$\Delta(a,b/K)=\Delta(a/K).$$\end{Fact}
Hence, the number $\Delta(a/K)$ is an invariant of the extension
$K(a)^{alg}/K$. Furthermore, if the difference subfield $L$ of $\calu$
is free from $K(a)$ over $K$, then $\Delta(a/L)=\Delta(a/K)$. From this
one easily obtains the following:

\begin{Corollary}\label{corsigma} Let $a,b$ be tuples in $\calu$, with
  $a$ and $b$ of transcendence
  degree $1$ over $K$, and such that $\si(a)\in K(a)^{alg}$,
  $\si(b)\in K(b)^{alg}$. Assume that $\Delta(a/K) 
\neq  
  \Delta(b/K)$. Then $tp(a/K)$ and $tp(b/K)$ are orthogonal. 
\end{Corollary}

Recall that the main theorem of \cite{Pi2} can be expressed as: if $K$ is a differential subfield of  ${\mathcal U}\models DCF_{0}$ which is algebraically closed and has a solution $B\in GL_{n}(K)$ of every linear differential equation $\delta (Z) = AZ$ over $K$, and $G$ is a linear differential algebraic group over $K$ then every differential algebraic $PHS$ over $K$ for $G$ has a   $K$-point. 

So the following  gives a counterexample to the analogous statement in our set-up.  The result should translate easily into a counterexample in the more difference ring and difference scheme-based set-up.

\begin{Proposition}
There is a difference subfield $K$ of ${\mathcal U}$ which is algebraically closed and strongly $PV$-closed  but for some linear difference algebraic group $G$ and difference-algebraic $PHS$ $X$ for $G$ over $K$, $X(K) = \emptyset$. 
\end{Proposition}
\begin{proof} 
We will take for $G$ the subgroup of $({\mathcal U}^{*},\times) =
GL_{1}({\mathcal U})$ defined by $\sigma(x) = x^{2}$,  rewritten as
$\sigma(x)/x^{2} = 1$.  
\\
 Fix generic $a\in {\mathcal U}$, that is $a$ is difference
 transcendental over $\emptyset$.
 There is a unique such type, which is moreover stationary. Consider the
 $a$-definable subset $X$: $\sigma(x) = ax^{2}$ of $({\mathcal
   U}^{*},\times)$.  It is a coset of $G$, hence a linear difference
 algebraic $PHS$ for $G$ defined over $K_{0}$ where $K_{0}$ is the difference subfield generated by $a$. 
Now $X$ is clearly a coset of $G$ in $GL_{1}(\mathcal U)$. If $b,c\in
X$, then both realise the generic type over $\emptyset$, hence they have
the same type over $a$, and $tp(b/K_0)=tp(c/K_0)$. Moreover, if $b\in X$,
then $\Delta(b/K_0)=1/2\neq 1$, and by Corollary \ref{corsigma}, $tp(b/K_0)$ is
orthogonal to the generic type of $Fix(\si)$. Hence it is orthogonal to
any type which is realised in a PV extension of $K_0$. \\
One can construct an extension $M$ of $K_0$ which is algebraically
closed, and closed under PV extensions with the following property:
$M=\bigcup_{\alpha<\kappa} M_\alpha$, where $M_0=K_0$,  $M_\alpha=\bigcup_{\beta<\alpha} M_\beta$
when $\alpha$ is a limit ordinal, and $M_{\alpha+1}=M_\alpha(c)^{alg}$,
where $c$ is a fundamental solution of some linear difference equation
$\sigma(Z) =  AZ$ over $M_\alpha$. By the above and using induction on $\alpha$, if $b\in X$, then
$tp(b/K_0)$ and $tp(c/M_\alpha)$ are orthogonal, so that $b\notin M_{\alpha+1}$.  Hence
$X(M)=\emptyset$.

\end{proof}

\end{document}